\documentclass[reqno]{amsart}

\usepackage{amsrefs}
\usepackage{hyperref}

\newtheorem{theorem}{Theorem}
\newtheorem{lemma}[theorem]{Lemma}

\newtheorem{corollary}[theorem]{Corollary}

\theoremstyle{exercise}

\theoremstyle{definition}
\newtheorem{definition}[theorem]{Definition}
\usepackage{graphicx}
\usepackage{pstricks, enumerate, pst-node, pst-text, pst-plot}

\numberwithin{equation}{section}
\numberwithin{theorem}{section}

\newcommand{\intav}[1]{\mathchoice {\mathop{\vrule width 6pt height 3 pt depth  -2.5pt
\kern -8pt \intop}\nolimits_{\kern -6pt#1}} {\mathop{\vrule width
5pt height 3  pt depth -2.6pt \kern -6pt \intop}\nolimits_{#1}}
{\mathop{\vrule width 5pt height 3 pt depth -2.6pt \kern -6pt
\intop}\nolimits_{#1}} {\mathop{\vrule width 5pt height 3 pt depth
-2.6pt \kern -6pt \intop}\nolimits_{#1}}}

\newcommand{\intavl}[1]{\mathchoice {\mathop{\vrule width 6pt height 3 pt depth  -2.5pt
\kern -8pt \intop}\limits_{\kern -6pt#1}} {\mathop{\vrule width 5pt
height 3  pt depth -2.6pt \kern -6pt \intop}\nolimits_{#1}}
{\mathop{\vrule width 5pt height 3 pt depth -2.6pt \kern -6pt
\intop}\nolimits_{#1}} {\mathop{\vrule width 5pt height 3 pt depth
-2.6pt \kern -6pt \intop}\nolimits_{#1}}}

\newcommand{\mc}{\mathcal}

\newcommand{\HH}{\mc{H}}

\newcommand{\R}{\mathbb{R}}
\newcommand{\N}{\mathbb{N}}

\newcommand{\pp}{\mathbb{P}}

\renewcommand{\P}[1]{{\mathbb{P}}\left[{#1}\right]}
\newcommand{\Psub}[2]{{\mathbb{P}}_{#1}\left[{#2}\right]}
\newcommand{\Prob}[2]{{\mathrm{Prob}}_{#1}\left[{#2}\right]}
\newcommand{\E}[1]{{\mathbb{E}}\left[{#1}\right]}
\newcommand{\EE}[2]{{\mathbb{E}}_{#1}\left[{#2}\right]}
\newcommand{\ind}[1]{{\bf 1}_{\left\{{#1}\right\}}}

\newcommand{\htime}{\tau}

\DeclareMathOperator{\supp}{supp}

\begin{document}

\title[An Abramov formula for stationary spaces of discrete groups]{
  An Abramov formula for stationary spaces of discrete groups}

\author[Yair Hartman]{Yair Hartman}
\address{Weizmann Institute of Science, Faculty of Mathematics and Computer Science, POB 26, 76100, Rehovot, Israel.}
\email{yair.hartman, yuri.lima, omer.tamuz@weizmann.ac.il}

\author[Yuri Lima]{Yuri Lima}

\author[Omer Tamuz]{Omer Tamuz}

\thanks{Yuri Lima is supported by the European Research Council, grant 239885.
Omer Tamuz is supported by ISF grant 1300/08, and is a
recipient of the Google Europe Fellowship in Social Computing, and
this research is supported in part by this Google Fellowship.}

\


\date{\today}


\begin{abstract}
  Let $(G,\mu)$ be a discrete group equipped with a generating
  probability measure, and let $\Gamma$ be a finite index subgroup of
  $G$. A $\mu$-random walk on $G$, starting from the identity, returns
  to $\Gamma$ with probability one. Let $\theta$ be the hitting
  measure, or the distribution of the position in which the random
  walk first hits $\Gamma$.

  We prove that the Furstenberg entropy of a $(G,\mu)$-stationary
  space, with respect to the induced action of $(\Gamma,\theta)$, is
  equal to the Furstenberg entropy with respect to the action of
  $(G,\mu)$, times the index of $\Gamma$ in $G$. The index is shown to
  be equal to the expected return time to $\Gamma$.

  As a corollary, when applied to the Furstenberg-Poisson boundary of
  $(G,\mu)$, we prove that the random walk entropy of
  $(\Gamma,\theta)$ is equal to the random walk entropy of $(G,\mu)$,
  times the index of $\Gamma$ in $G$.
\end{abstract}

\maketitle

\section{Introduction and statement of results}

Let $G$ be a countable, discrete group, equipped with a generating
probability measure $\mu$, and consider the $\mu$-random walk on $G$:
at each step, the increments are chosen independently according to the
measure $\mu$. We study recurrent subgroups: $\Gamma$ is a recurrent
subgroup of the pair $(G,\mu)$ if the $\mu$-random walk on $G$ almost
surely visits $\Gamma$ infinitely often.

A natural measure on a recurrent subgroup $\Gamma$ is the hitting
measure $\theta$, which is the distribution of the first element of
$\Gamma$ visited by a $\mu$-random walk on $G$. Our main result
answers a question of Furstenberg regarding the relation between the
$\mu$-random walk on $G$ and the corresponding $\theta$-random walk on
$\Gamma$, from the point of view of their stationary actions.

The main objects of study of this article are stationary spaces. A
probability space $(X,\nu)$ is $(G,\mu)$-stationary if it is a
measurable $G$-space, and the convolution $\mu*\nu$ is equal to
$\nu$. Essentially, the action of $G$ on $X$ leaves $\nu$ invariant
``on the average'', when this average is taken over $\mu$.

A $G$-space is in particular a
$\Gamma$-space. Furthermore, it is known that if $(X,\nu)$ is
$(G,\mu)$-stationary then it is also $(\Gamma,\theta)$-stationary.

The Furstenberg entropy of a $(G,\mu)$-stationary space $(X,\nu)$,
denoted by $h_\mu(X,\nu)$, measures the average deformation of $\nu$
under the action of $G$. We prove that the entropies of $(X,\nu$) with
respect to the actions of $G$ and of $\Gamma$ are easily related.

\begin{theorem}[Abramov formula for stationary spaces]\label{thm abramov}
  Let $G$ be a countable, discrete group with a generating probability
  measure $\mu$ with finite entropy. Let $\Gamma$ be a finite index
  subgroup of $G$, with hitting measure $\theta$. Then, for any
  $(G,\mu)$-stationary space $(X,\nu)$,
  \begin{align}
    \label{eq abramov intro}
    h_\theta(X,\nu)=[G:\Gamma]\cdot h_\mu(X,\nu).
  \end{align}
\end{theorem}

In classical ergodic theory, if $T:(X,\nu)\rightarrow (X,\nu)$ is a
measure-preserving transformation and $A\subset X$ has positive
measure, one can define a measure-preserving induced transformation on
$A$ as the first return of $x\in A$ to $A$, under repeated
applications of $T$.  Abramov related in~\cite{abramov1966entropy} the
entropy of $T$ with the entropy of the induced map.  More
specifically, he proved that the entropy of the induced map is equal
to $1/\nu(A)$ times the entropy of the initial map, thus relating the
ratio of entropies to the fraction of $X$ occupied by $A$. If one
thinks of the index $[G:\Gamma]$ as the portion that $\Gamma$ occupies
in $G$, then Theorem~\ref{thm abramov} makes the analogous
statement. We therefore call Theorem~\ref{thm abramov} an {\it Abramov
  formula}.

An ingredient of the proof of Theorem~\ref{thm abramov} is the
following result, which states that the expected return time of a
$\mu$-random walk to $\Gamma$ is equal to its index in $G$. In
particular, the expected return time is independent of the measure
$\mu$. This holds in the more general setting of topological groups
and open subgroups.

\begin{theorem}[Kac formula for subgroups]\label{thm kac}
  Let $G$ be a second countable topological group with a generating
  probability measure $\mu$, and let $\Gamma$ be an open, recurrent
  subgroup of $(G,\mu)$. Then the expected return time $\E{\htime}$ of
  the $\mu$-random walk to $\Gamma$ is equal to the index of $\Gamma$
  in $G$:
  \begin{align*}
    \E{\htime} = [G:\Gamma].
  \end{align*}
\end{theorem}

We call Theorem~\ref{thm kac} a {\it Kac formula} in another analogy
with classical ergodic theory: there, Kac's formula states that the
expected return time of $x \in A$ to $A$, under repeated applications
of $T$, is equal to $1/\nu(A)$.  We indeed use the classical Kac
formula to establish this theorem.

Finally, we apply Theorem~\ref{thm abramov} to the special case that
$(X,\nu)$ is the Furstenberg-Poisson boundary of $(G,\mu)$, which is
naturally isomorphic to the Furstenberg-Poisson boundary of
$(\Gamma,\theta)$.  Using the Kaimanovich-Vershik entropy
characterization of the Furstenberg-Poisson
boundary~\cite{kaimanovich1983random}, we get that $h(\Gamma,\theta)$,
the random walk entropy of $(\Gamma,\theta)$, is equal to the index of
$\Gamma$ in $G$ times $h(G,\mu)$, the random walk entropy of
$(G,\mu)$.

\begin{corollary}\label{corollary}
  Let $G$ be a countable, discrete group with a generating probability
  measure $\mu$ with finite entropy. Let $\Gamma$ be a finite index
  subgroup of $G$, with hitting measure $\theta$. Then
  \begin{align}
    \label{eq kac intro}
    h(\Gamma,\theta)=[G:\Gamma] \cdot h(G,\mu).
  \end{align}
\end{corollary}

Note that by Theorem~\ref{thm kac} we could write Eq.~\ref{eq kac
  intro} as
\begin{align*}
  h(\Gamma,\theta)=\E{\htime} \cdot h(G,\mu).
\end{align*}
Likewise, Eq.~\ref{eq abramov intro} can be written as
\begin{align*}
  h_\theta(X,\nu)=\E{\htime}\cdot h_\mu(X,\nu).
\end{align*}

The paper is organized as follows. In Section~\ref{section
  preliminaries} we introduce the basic notations and definitions as
well as the necessary background for the sequel. Section~\ref{section
  hitting time=index} is devoted to the proof of Theorem~\ref{thm
  kac}, and Section~\ref{section abramov formula} to the proof of
Theorem \ref{thm abramov}. To this matter, Appendix~\ref{appendix
  probabilistic lemma} treats the required results of Markov chains,
adapted to our context.

\section{Preliminaries}
\label{section preliminaries}

The following definitions are mostly standard. The notation is adapted
from Furman~\cite{furman2002random}, who provides an excellent
exposition to random walks on groups and Furstenberg-Poisson boundary
theory.

\subsection{Random walks on groups and random walk entropy}
\label{sub random walks}

Let $G$ be a countable, discrete group with identity $e$, and let
$\mu$ be a {\it generating} probability measure on $G$, so that the
semigroup generated by its support,
\begin{align*}
  \supp(\mu) = \{g\in G\,;\,\mu(g)>0\},
\end{align*}
is equal to the whole of $G$. Let $X_1,X_2,\ldots$ be a sequence of
$G$-valued independent random variables each with law $\mu$.  The
Markov chain $\{Z_n\}_{n=1}^\infty$ defined as
\begin{align*}
  Z_n = X_1X_2 \cdots X_n
\end{align*}
is called the {\it $\mu$-random walk} on $G$.  Formally, if we let
$(\Omega,\pp)=(G,\mu)^\N$, then
\begin{equation*}
  \begin{array}{rcrcl}
    X_n&:&(\Omega,\pp)&\longrightarrow &G\\
    & & \omega&\longmapsto     &X_n(\omega)=\omega_n,
\end{array}
\end{equation*}
and
\begin{equation*}
  \begin{array}{rcrcl}
    Z_n&:&(\Omega,\pp)&\longrightarrow &G\\
    & & \omega&\longmapsto     &Z_n(\omega)=X_1(\omega)\cdots X_n(\omega).
\end{array}
\end{equation*}
We will occasionally consider random walks starting from some $g \in
G$, in which case $Z_n = gX_1X_2 \cdots X_n$.

Let $\mu$ be a probability measure on $G$. The standard definition for
the {\it entropy} of $\mu$ is
\begin{align*}
  H(\mu)=-\sum_{g\in G}\mu(g)\cdot\log{\mu(g)},
\end{align*}
where $0\,\cdot\,\log 0=0$.  Throughout this work, we only consider
finite entropy probability measures, i.e., measures for which $H$ is
finite. As Kaimanovich and Vershik point out, measures with infinite
entropy are somewhat different in nature (see the remark on page 465
in~\cite{kaimanovich1983random}).

Let $\mu_1$ and $\mu_2$ be probability measures on $G$. The {\it
  convolution} of $\mu_1$ and $\mu_2$ is defined by
\begin{align*}
  [\mu_1*\mu_2](g) = \sum_{g' \in G}\mu_1(g')\mu_2(g'^{-1}g).
\end{align*}
Equivalently, $\mu_1*\mu_2$ is the push-forward of the product measure
$\mu_1\times\mu_2$, under the product map $(g_1,g_2)\mapsto g_1g_2$
from $G\times G$ to $G$. By standard information theoretical
inequalities it holds that
\begin{align*}
  H(\mu_1*\mu_2) \le H(\mu_1)+H(\mu_2).
\end{align*}
In particular, for a probability measure $\mu$ on $G$, if $\mu^n$
denotes the $n$th convolution of $\mu$ with itself, the sequence
$H(\mu^n)/n$ converges, and so we can consider the following
definition of random walk entropy~\cite{avez1972entropie}.
\begin{definition}
  The {\it random walk entropy}, also known as {\it Avez entropy} or
  {\it asymptotic entropy}, of the pair $(G,\mu)$ is
  \begin{align*}
    h(G,\mu) = \lim_{n\rightarrow\infty}\dfrac{1}{n}H(\mu^n).
  \end{align*}
\end{definition}
The measure $\mu^n$ is the distribution of the position of a
$\mu$-random walk in its $n$th step.  In a sense, $h(G,\mu)$ measures
the rate of escape to infinity of a $\mu$-random walk on $G$.

In this article, we are interested in the relation between the
$\mu$-random walk on $G$ and its induced random walk on a recurrent
subgroup.
\begin{definition}
  A subgroup $\Gamma$ of $G$ is {\it $\mu$-recurrent} if, for
  $\pp$-almost every $\omega\in\Omega$, there exists an $n \geq 1$
  ($\Leftrightarrow$ infinitely many values of $n$) such that
  $Z_n(\omega)\in\Gamma$.
\end{definition}
The recurrence property is equivalent to the existence of a {\it
  return time map} $\htime:\Omega\rightarrow\N$ defined $\pp$-almost
everywhere by
\begin{align*}
  \htime(\omega)=\min\{n \ge 1\,;\,Z_n(\omega)\in\Gamma\}
\end{align*}
and thus of a {\it hitting map}
\begin{equation*}
\begin{array}{ccrcl}
\Phi&:&\Omega&\longrightarrow &\Gamma\\
    & &\omega&\longmapsto     &Z_{\htime(\omega)}(\omega).
\end{array}
\end{equation*}
The random variable $\Phi$ is the first element of $\Gamma$ hit by the
random walk.  A useful related definition is that of avoiding sets
$A_n$.
\begin{definition}
  \label{def avoiding}
  Let $\Gamma$ be a subgroup of $G$. The $n$th {\it avoiding set}
  $A_n$ is
  \begin{equation}\label{eq 6}
    A_n=\left\{(g_1,\ldots,g_n)\in G^n\,;\,g_1\cdots
      g_k\not\in\Gamma\text{ for }k \leq n\right\}.
  \end{equation}
\end{definition}
Equivalently, $A_n$ is the set of all length $n$ walks on $G$ that do
not hit $\Gamma$. Evidently,
\begin{align*}
  \htime > n\quad\Leftrightarrow\quad(X_1,\ldots,X_n) \in A_n.
\end{align*}

Let $\Gamma$ be a $\mu$-recurrent subgroup of $G$. The distribution of
$\Phi$ defines a natural probability measure on $\Gamma$, called the
hitting measure.
\begin{definition}
  Let $\Gamma$ be a $\mu$-recurrent subgroup of $G$. The {\it hitting
    measure} of $\Gamma$ is the probability measure $\theta$ on
  $\Gamma$ defined by
  \begin{align*}
    \theta(\gamma) = \P{\Phi=\gamma}\,,\ \ \gamma \in \Gamma.
  \end{align*}
\end{definition}
Equivalently, $\theta$ is equal to the push-forward of $\pp$ under
$\Phi$. Another description, which will be useful for our purposes, is
the following: for each $n\ge 1$ define $\theta^{(n)}(\gamma)$ to be
the probability that a $\mu$-random walk on $G$ will hit $\Gamma$
first at step $n$, exactly at $\gamma\in\Gamma$:
\begin{align*}
  \theta^{(n)}(\gamma) = \P{(X_1,\ldots,X_{n-1}) \in A_{n-1}\mbox{ and
    } Z_n=\gamma}.
\end{align*}
Note that $\theta^{(n)}$ is not necessarily a probability measure, and
that
\begin{align*}
  \theta = \sum_{n\in\mathbb{N}} \theta^{(n)}.
\end{align*}

With respect to the hitting time map $\htime$, one can divide
recurrent subgroups into two classes.
\begin{definition}
The {\it expected return time} to $\Gamma$ is equal to
\begin{align*}
  \E{\htime}=\int_{\Omega}\htime(\omega)d\pp(\omega).
\end{align*}
$\Gamma$ is called {\it positive recurrent} if $\E{\htime}<\infty$ and
{\it null recurrent} if $\E{\htime}=\infty$.
\end{definition}

\subsection{Stationary spaces and Furstenberg entropy}\label{sub stationary spaces}

Let $X$ be a {\it measurable $G$-space}. By this we mean that $G$ acts on $X$
respecting the group operation of $G$, and that the action map
\begin{equation*}
\begin{array}{rcl}
G\times X&\longrightarrow &X\\
(g,x)    &\longmapsto     &gx
\end{array}
\end{equation*}
is measurable. Let $\nu$ be a probability measure on $X$.
For $g \in G$ we denote by $g\nu$ the probability measure on $X$ defined by
\begin{align*}
  g\nu(A) = \nu(g^{-1}A),
\end{align*}
where $A$ is a measurable subset of $X$.

The {\it convolution} of the measures $\mu$ on $G$ and $\nu$ on $X$,
denoted by $\mu*\nu$, is the probability measure on $X$ defined as the
image of the product measure $\mu\times\nu$ under the above action
map. Equivalently,
\begin{align*}
  \mu * \nu = \sum_{g \in G} \mu(g) \cdot g\nu.
\end{align*}

\begin{definition}
  A probability measure $\nu$ on a $G$-space $X$ is called {\it
    $\mu$-stationary} if $\mu*\nu=\nu$.  In this case, $(X,\nu)$ is
  called a {\it $(G,\mu)$-stationary space}.
\end{definition}

This can be interpreted as saying that $\nu$ is invariant ``on the
average''.  Every such stationary action preserves the measure class
of $\nu$, that is, $\nu(A)=0$ if and only if $g\nu(A)=0$ for every
$g\in G$ and measurable $A \subset X$ (see, e.g., Nevo and
Zimmer~\cite{nevozimmer1999}). In this case, the {\it Radon-Nikodym
  cocycle} $\rho:G\times X\rightarrow\R$ is defined as
\begin{equation}\label{eq 3}
  \rho(g,x)=-\log\dfrac{dg^{-1}\nu}{d\nu}(x).
\end{equation}
Note that $\rho$ indeed satisfies the additive cocycle relation:
\begin{equation}\label{eq 4}
  \rho(gg_1,x)=\rho(g,g_1x)+\rho(g_1,x).
\end{equation}
Define $\varphi:G\rightarrow\R$ by
\begin{equation}\label{eq 5}
  \varphi(g)=\int_X\rho(g,x)d\nu(x)
\end{equation}
which, by Jensen's inequality, is nonnegative and equal to zero if and
only if $g^{-1}\nu=\nu$.  Unlike the measure-preserving
transformations of classical ergodic theory, here each $g\in G$ may
deform $\nu$, and this deformation is quantified by $\varphi(g)$. In
terms of information theory, $\varphi(g)$ is the {\it Kullback-Leibler
  divergence} $D_{KL}(g\nu||\nu)$ between the measures $g\nu$ and
$\nu$.  The average deformation is measured as follows.

\begin{definition}
  The {\it Furstenberg entropy} of a $(G,\mu)$-stationary space
  $(X,\nu)$ is
  \begin{align*}
    h_\mu(X,\nu)=\sum_{g\in G}\mu(g)\cdot\varphi(g)=\sum_{g\in G}\mu(g)\cdot\int_X-\log\dfrac{dg^{-1}\nu}{d\nu}d\nu(x)d\mu(g).
  \end{align*}
\end{definition}

The Furstenberg entropy and the random walk entropy relate to each
other in the following way.

\begin{theorem}[Kaimanovich and Vershik~\cite{kaimanovich1983random}]\label{thm compare entropies}
  If $(X,\nu)$ is a $(G,\mu)$-stationary space, then
  \begin{align*}
    h_\mu(X,\nu)\le h(G,\mu),
  \end{align*}
  with equality if $(X,\nu)$ is the Furstenberg-Poisson boundary of
  $(G,\mu)$.
\end{theorem}

The Furstenberg-Poisson boundary of $(G,\mu)$ is introduced in the
next section.

\subsection{Furstenberg-Poisson boundary}\label{sub furstenberg-poisson boundary}

\begin{definition} A function $h:G\rightarrow\R$ is {\it
    $\mu$-harmonic} if it satisfies the $\mu$-mean value property
  \begin{align*}
    h(g)=\sum_{g_1 \in G}\mu(g_1) \cdot h(gg_1)
  \end{align*}
  for all $g \in G$.
\end{definition}
We will call a function {\it harmonic}, without explicit reference to
the measure, whenever the measure is obvious from the context.

Let $\HH^\infty(G,\mu)$ be the space of all bounded harmonic functions
on $G$ and $L^\infty(X,\nu)$ be the space of bounded functions on $X$,
with respect to the measure class of $\nu$. When $(X,\nu)$ is
$(G,\mu$)-stationary, to each $f \in L^\infty(X,\nu)$ we associate the
bounded harmonic function $F_\mu(f)\in\HH^\infty(G,\mu)$ given by
\begin{equation}\label{eq 1}
  [F_\mu(f)](g)=g\nu(f)=\int_X f(gx)d\nu(x)\,,\ \ \ g\in G.
\end{equation}
This defines a linear map $F_\mu$ from $L^\infty(X,\nu)$ to
$\HH^\infty(G,\mu)$, called the {\it Furstenberg transform}. The next
lemma shows that the image of $F_\mu$ is indeed in
$\HH^\infty(G,\mu)$, and that this condition is equivalent to the
$(G,\mu)$-stationarity of $(X,\nu)$.

\begin{lemma}\label{lemma stationarity}
  Let $(X,\nu)$ be a $G$-space and the Furstenberg transform map
  $F_\mu:L^\infty(X,\nu)\rightarrow L^\infty(G,\mu)$ defined as in
  Eq.~\ref{eq 1}. Then
  $F_\mu(L^\infty(X,\nu))\subset\HH^\infty(G,\mu)$ if and only if
  $(X,\nu)$ is $(G,\mu)$-stationary.
\end{lemma}

\begin{proof}
  First assume that
  $F_\mu(L^\infty(X,\nu))\subset\HH^\infty(G,\mu)$. Let $f\in
  L^\infty(X,\nu)$ and $h=F_\mu(f)$. Then
  \begin{align*}
    [\mu * \nu] (f) = \sum_{g\in G} \mu(g) \cdot g\nu(f) = \sum_{g\in
      G} \mu(g) \cdot h(g) = h(e) =\nu(f),
  \end{align*}
  where in the third equality we used the $\mu$-harmonicity of $h$ at
  $e$.  Because $f$ is arbitrary, it follows that
  $\mu*\nu=\nu$.

  Conversely, if there exists an $f_0$ such that $h_0=F_\mu(f_0)$ is not
  harmonic at some $g_0 \in G$ then, assuming without loss of
  generality that $g_0=e$,
  \begin{align*}
    [\mu * \nu] (f_0) = \sum_{g\in G}\mu(g)\cdot g\nu(f_0)
    = \sum_{g\in G}\mu(g)\cdot h_0(g)\neq h_0(e) = \nu(f_0),
  \end{align*}
  and so $\mu*\nu \neq \nu$.
\end{proof}

Among the many characterizations of the Furstenberg-Poisson boundary,
we consider the following.

\begin{definition}
  The {\it Furstenberg-Poisson boundary} of $(G,\mu)$ is the unique
  $\mu$-stationary $G$-space $(X,\nu)$ for which the Furstenberg
  transform $F_\mu$ is an isometric bijection from $L^\infty(X,\nu)$
  to $\HH^\infty(G,\mu)$.
\end{definition}

The aforementioned uniqueness is up to isomorphism of
$(G,\mu)$-stationary spaces. For precise definitions and constructions, we refer the
reader to Bader and Shalom~\cite{bader2006factor}  or Furstenberg and
Glasner~\cite{furstenberg2009stationary}.

\subsection{Induced action}

Let $(X,\nu)$ be a $(G,\mu)$-stationary space and $\Gamma$ a
$\mu$-recurrent subgroup of $G$ with hitting measure $\theta$. In this
section it is shown that the restricted $\Gamma$-action on $X$ is
$\theta$-stationary.  This follows from the stronger fact that there
exists a one-to-one correspondence between $\mu$-harmonic functions on
$G$ and $\theta$-harmonic functions on $\Gamma$, according to Theorem
\ref{thm furstenberg} below.  In particular, $(G,\mu)$ and
$(\Gamma,\theta)$ have isomorphic Furstenberg-Poisson
boundaries. These results are due to
Furstenberg~\cite{furstenberg1971random}. We include the proof for
completeness.

If $\Gamma$ is recurrent then it is also recurrent for a random walk
starting from an arbitrary $g \in G$: since $\mu$ is generating, the
random walk hits $g$ with positive probability, and conditioned on
that returns to $\Gamma$ with probability one. Hence a random walk
that starts at $g$ will return to $\Gamma$ with probability one.

Accordingly, let $\theta_g$ denote the hitting measure on $\Gamma$ of
the $\mu$-random walk starting at $g \in G$, and let
$\theta_g^{(n)}(\gamma)$ be the probability that a $\mu$-random walk
starting at $g$ will hit $\Gamma$ first at step $n$, exactly at
$\gamma$.

\begin{lemma}
  \label{lemma gamma boundary}
  Let $h\in\HH^\infty(G,\mu)$. Then
  \begin{align*}
    h(g) = \theta_g(h)\,,\ \ \mbox{for all } g \in G.
  \end{align*}
\end{lemma}

The intuition behind this lemma is the following: $\Gamma$ is
recurrent, and hence the $\mu$-random walk hits $\Gamma$ almost
surely. $\Gamma$ can be therefore be viewed as a boundary of the
random walk, and so, as in the solution of the classical Dirichlet
problem, the value of a harmonic function $h$ at some $g \in G$ is
equal to the average value of $h$ on this boundary, weighted according
to the hitting measure of a random walk starting at $g$. That is,
$h(g)=\theta_g(h)$. In particular, $h$ is determined by its values on
$\Gamma$.

\begin{proof}
  Since $h$ is a bounded harmonic function, the sequence of random variables
  $M_1,M_2,\ldots$ defined by
  \begin{align*}
    M_n=h(Z_n)
  \end{align*}
  is a bounded martingale. Note that $\htime$ is a stopping time, and hence,
  by the optional stopping theorem,
  \begin{align*}
    h(g)=\EE{g}{h(Z_1)}=\EE{g}{h(Z_\htime)}=\theta_g(h),
  \end{align*}
  where $\EE{g}{\cdot}$ denotes expectation for random walks starting
  at $g$.
\end{proof}

The next theorem shows that not only is $h$ determined by its values
on $\Gamma$, but that the restriction of $h$ to $\Gamma$ is an
isometric bijection between $\mu$-harmonic functions on $G$ and
$\theta$-harmonic functions on $\Gamma$.

\begin{theorem}[Furstenberg \cite{furstenberg1971random}]\label{thm furstenberg}
The restriction map
\begin{align}
  \begin{array}{ccrcl}
    \Psi&:&\HH^\infty(G,\mu)&\longrightarrow &\HH^\infty(\Gamma,\theta)\\
    &&h&\longmapsto&h|_\Gamma
  \end{array}\label{eq 11}
\end{align}
is an isometric bijection between $\HH^\infty(G,\mu)$ and
$\HH^\infty(\Gamma,\theta)$.
\end{theorem}

\begin{proof}
  First we show that the image of $\Psi$ is in
  $\HH^\infty(\Gamma,\theta)$: for $h\in\HH^\infty(G,\mu)$, we show
  that $\Psi(h)$ is $\theta$-harmonic. Let $h' = \Psi(h)$, so
  that $h'(\gamma)=h(\gamma)$ for $\gamma \in \Gamma$. By
  Lemma~\ref{lemma gamma boundary} above, we have that
  \begin{align*}
    h'(\gamma) = h(\gamma) = \theta_\gamma(h) = \sum_{\lambda \in
      \Gamma}\theta_\gamma(\lambda)\cdot h(\lambda).
  \end{align*}
  Since in the last expression we evaluate $h$ only on $\Gamma$ then
  we can replace $h$ with $h'$, so that
  \begin{align*}
    h'(\gamma) = \sum_{\lambda \in \Gamma}\theta_\gamma(\lambda)\cdot h'(\lambda).
  \end{align*}

  We claim that
  \begin{align}
    \label{eq theta-gamma}
    \theta_\gamma(\lambda) =  \theta(\gamma^{-1}\lambda).
  \end{align}
  To see this, couple two $\mu$-random walks to perform the same
  increments, where one starts at $\gamma$ and the other at the
  identity. Then the first walk visits $\lambda$ exactly when the
  second walk visits $\gamma^{-1}\lambda$.

  Hence
  \begin{align*}
    h'(\gamma) = \sum_{\lambda \in \Gamma}\theta(\lambda)\cdot h'(\gamma\lambda),
  \end{align*}
  and so $h'=\Psi(h)$ is indeed $\theta$-harmonic.

  We next show that $\Psi$ is a bijection. By Lemma~\ref{lemma gamma
    boundary} $h$ is determined by its values on $\Gamma$, and
  therefore $\Psi$ is one-to-one. To show that $\Psi$ is onto, given
  $h' \in \HH^\infty(\Gamma,\theta)$, define $h \in \HH^\infty(G,\mu)$
  by $h(g) = \theta_g(h')$. Observe that $\Psi(h)=h'$: for $\gamma \in
  \Gamma$,
  \begin{align*}
    h(\gamma) = \theta_\gamma(h') = \sum_{\lambda \in
      \Gamma}\theta_\gamma(\lambda)\cdot h'(\lambda) = h'(\gamma),
  \end{align*}
  by Eq.~\ref{eq theta-gamma} and the $\theta$-harmonicity of
  $\gamma$. We now show that $h$ is indeed $\mu$-harmonic.

  Express the probability to first hit $\Gamma$ at $\gamma$ when
  starting a $\mu$-random walk at $g$ using the law of conditional
  probabilities, by conditioning on the first step, and separating the
  sum to the events that the first step either hit or did not hit
  $\Gamma$:
  \begin{align*}
    \theta_{g}(\gamma) = \sum_{gg_1 \in
      \Gamma}\mu(g_1) \cdot \delta_{gg_1}(\gamma)  + \sum_{gg_1 \not \in
      \Gamma}\mu(g_1) \cdot \theta_{gg_1}(\gamma) ,
  \end{align*}
  where $\delta_g$ is the Dirac measure concentrated on $g$.
  Hence we can write this equality as an equality
  of measures:
  \begin{align*}
    \theta_{g} = \sum_{gg_1 \in
      \Gamma} \mu(g_1) \cdot \delta_{gg_1} + \sum_{gg_1 \not \in
      \Gamma}\mu(g_1) \cdot \theta_{gg_1}.
  \end{align*}
  But for $gg_1 \in \Gamma$ it holds that $\delta_{gg_1}(h') =
  h'(gg_1) = \theta_{gg_1}(h')$. Hence
  \begin{align*}
    \theta_{g}(h') = \sum_{g_1 \in G}\mu(g_1) \cdot \theta_{gg_1}(h').
  \end{align*}
  The left hand side of the equation above is equal to $h(g)$, and the
  right hand side is the $\mu$-mean value property of $h$ at $g$, and
  so $h$ is indeed $\mu$-harmonic.

  To conclude the proof we show that $\Psi$ preserves the sup
  norms. Clearly, $\lVert h' \rVert \le \lVert h \rVert$. Now if
  $\lVert h' \rVert < \lVert h \rVert$, then there exists some $g_0\in
  G$ with $|h(g_0)|>|h(\gamma)|$ for all $\gamma\in\Gamma$. But
  $h(g_0)=\theta_{g_0} (h)$, that is $h(g_0)$ is an average of values
  of the form $h(\gamma)$, in contradiction.
\end{proof}

\begin{corollary}\label{corollary poisson boundary}
  Every $(G,\mu)$-stationary space is also
  $(\Gamma,\theta)$-stationary.  Furthermore, $(G,\mu)$ and
  $(\Gamma,\theta)$ have the same Furstenberg-Poisson boundary.
\end{corollary}

\begin{proof}
  If $(X,\nu)$ is $(G,\mu)$-stationary, then $F_\theta=\Psi\circ
  F_\mu$ maps $L^\infty(X,\nu)$ to $\HH^\infty(\Gamma,\theta)$. By
  Lemma \ref{lemma stationarity}, this means that $\nu$ is
  $\theta$-stationary.

  If $(X,\nu)$ is the Furstenberg-Poisson boundary of $(G,\mu)$, then
  $F_\theta=\Psi\circ F_\mu$ is a composition of
  isometric bijections and thus an isometric bijection as well.
\end{proof}

\section{Expected hitting time and index: a Kac formula}\label{section hitting time=index}

The goal of this section is to prove Theorem~\ref{thm kac}, namely
that the expected return time to a recurrent subgroup is equal to its
index, whether the subgroup is positive or null recurrent.  We do this
by inducing a Markov chain on the quotient $\Gamma\backslash G$.

Let $\Gamma$ be a subgroup of $G$,
\begin{align*}
  \Gamma\backslash G=\{\Gamma g\,;\,g\in G\}
\end{align*}
the set of right cosets and $\pi:G\rightarrow \Gamma\backslash G$ the
projection map. $G$ naturally acts on $\Gamma\backslash G$ by right
multiplication,
\begin{equation*}
\begin{array}{rcl}
\Gamma\backslash G\times G&\longrightarrow &\Gamma\backslash G\\
(\Gamma g,g_1)    &\longmapsto     &\Gamma gg_1\,.
\end{array}
\end{equation*}
Since $((\Gamma h)g)g_1=(\Gamma h)(gg_1)$, each map
\begin{equation*}
\begin{array}{ccrcl}
g&:&\Gamma\backslash G&\longrightarrow &\Gamma\backslash G\\
& &          \Gamma g_1&\longmapsto     &(\Gamma g_1)g=\Gamma g_1g
\end{array}
\end{equation*}
is a bijection.

The Markov chain $\{Z_n\}_{n=1}^\infty$ projects under $\pi$ to a
Markov chain $\{Y_n\}_{n=1}^\infty$ which has special properties,
according to the following lemma.

\begin{lemma}\label{lemma markov chain}
  Let $G$ be a second countable topological group with a generating
  probability measure $\mu$, and let $\Gamma$ be an open, recurrent
  subgroup of $(G,\mu)$. Then $\{Y_n\}_{n=1}^\infty$ is a time-independent Markov chain
  on $\Gamma\backslash G$. Furthermore, it is
  \begin{enumerate}[(a)]
  \item doubly stochastic,
  \item irreducible, and
  \item recurrent.
  \end{enumerate}
\end{lemma}

\begin{proof}
  Assume first that $G$ is discrete. For $\Gamma g_1,\Gamma
  g_2\in\Gamma\backslash G$, denote the transition probabilities by
  \begin{align*}
    p(\Gamma g_1,\Gamma g_2)=\sum_{g\in G\atop{\Gamma g_1g=\Gamma g_2}}\mu(g).
  \end{align*}
  This indeed defines a stochastic matrix ${\bf P}=(p(\Gamma g_1,\Gamma
  g_2))_{\Gamma\backslash G\times \Gamma\backslash G}$ for which
  \begin{align*}
    \P{Y_n=\Gamma g_n\,|\,Y_1=\Gamma g_1,\ldots,Y_{n-1}=\Gamma g_{n-1}}
    &=\P{Y_n=\Gamma g_n\,|\,Y_{n-1}=\Gamma g_{n-1}}\\
    &= p(\Gamma g_{n-1},\Gamma g_n).
  \end{align*}
  Hence $\{Y_n\}_{n=1}^\infty$ is a time-independent Markov chain on
  $\Gamma\backslash G$ with transition matrix ${\bf P}$.
  We now  prove (a), (b) and (c).
  \begin{enumerate}[(a)]
  \item The sum of the column of ${\bf P}$ associated with $\Gamma g_2$ is
    \begin{align*}
      \sum_{\Gamma g_1\in\Gamma\backslash G}p(\Gamma g_1,\Gamma g_2)
      &=\sum_{\Gamma g_1\in\Gamma\backslash G}\sum_{g\in G\atop{\Gamma
          g_1g=\Gamma g_2}}\mu(g).
    \end{align*}
    Changing the order of summation and rearranging we get
    \begin{align*}
      \sum_{\Gamma g_1\in\Gamma\backslash G}p(\Gamma g_1,\Gamma g_2)
      &=\sum_{g\in G}\mu(g)\cdot |\{\Gamma g_1\in\Gamma\backslash G\,;\,\Gamma g_1g=\Gamma g_2\}|=\sum_{g\in G}\mu(g)=1,
    \end{align*}
    where in the second equality we used that each
    $g:\Gamma\backslash G\rightarrow \Gamma\backslash G$ is a
    bijection.
  \item The fact that $\mu$ is generating guarantees that
    $\{Z_n\}_{n=1}^\infty$ is an irreducible Markov chain, and the
    irreducibility property descends to $\{Y_n\}_{n=1}^\infty$.
  \item Note that $Z_n$ belongs to $\Gamma$ if and only if $Y_n =
    \Gamma e$, the trivial coset. Therefore, since
    $\{Z_n\}_{n=1}^\infty$ is recurrent to $\Gamma$,
    $\{Y_n\}_{n=1}^\infty$ is recurrent to the trivial coset $\Gamma
    e$.
  \end{enumerate}

  Consider now the general case that $\Gamma$ is an open, recurrent
  subgroup of a second countable group $G$.  Then the quotient space
  $\Gamma\backslash G$ is countable, and so we can define the Markov
  chain $\{Y_n\}_{n=1}^\infty$ as above, substituting integrals of
  $\mu$ for sums of $\mu$. The
  assumption that $\mu$ is generating means in this context that the
  semigroup generated by the support of $\mu$,
  \begin{align*}
    \supp(\mu) = \{g\in G\,;\,\mu(A)>0\text{ for any open subset }A\text{ containing }g\},
  \end{align*}
  is the whole of $G$. This implies that $\{Y_n\}_{n=1}^\infty$ is
  irreducible. The proof that it is doubly stochastic and recurrent is
  identical to the proof above.
\end{proof}

Continuing the discussion of item (c) above, the return time $\htime$
of $\{Z_n\}_{n=1}^\infty$ to $\Gamma$ is equal to the return time
\begin{equation*}
\begin{array}{ccrcl}
\overline\htime&:&\Omega&\longrightarrow &\N\\
             & &\omega&\longmapsto     &\min\{n\ge 1\,;\,Y_n(\omega)=\Gamma e\}
\end{array}
\end{equation*}
of $\{Y_n\}_{n=1}^\infty$ to the trivial coset $\Gamma e$. In
particular, $\E{\overline\htime} =\E{\htime}$ and thus
$\{Y_n\}_{n=1}^\infty$ is a positive/null recurrent Markov chain if
and only if $\Gamma$ is a positive/null recurrent subgroup.

We are now ready to prove Theorems~\ref{thm kac}. In this proof we use
some classical results on Markov chains, which the uninitiated reader
may find in textbooks such as~\cite{levin2009markov}.
\begin{proof}[Proof of Theorem \ref{thm kac}]
  Consider the definition of $\{Y_n\}_{n=1}^\infty$ above.  Since
  $\{Y_n\}_{n=1}^\infty$ is irreducible and recurrent, it has a unique
  (up to multiplication by constants) stationary measure $\eta$. Since
  it is also doubly stochastic, this stationary measure is constant.

  If $[G:\Gamma] < \infty$ then we can normalize $\eta$ to be a
  probability measure, in which case, by Kac's theorem, $\E{\htime}$,
  the expected return time to $\Gamma e$, is equal to $1/\eta(\Gamma
  e) = [G:\Gamma]$. If $[G:\Gamma] = \infty$ then $\eta$ cannot be
  normalized to be a probability measure. Hence $\{Y_n\}_{n=1}^\infty$
  admits no stationary probability measure, and is therefore null
  recurrent, so that $\E{\htime} = \infty$.

\end{proof}

As a direct consequence of this Kac formula, the property of positive
recurrence of a subgroup $\Gamma$ of $G$ is independent of $\mu$: for
any generating measure, the positive recurrent subgroups of $G$ are
those with finite index. On the other hand, the properties of null
recurrence and transience do depend on the chosen measure $\mu$. For
example, let $G=\mathbb Z$ and $\Gamma=\{0\}$ be the trivial
subgroup. Then $\Gamma$ is recurrent for the simple random walk on
$\mathbb Z$, but transient for any random walk on $\mathbb Z$ with
drift. It may be interesting to characterize the subgroups that are
transient for any generating measure $\mu$ on $G$.

\section{An Abramov formula}\label{section abramov formula}

In this section we prove Theorem \ref{thm abramov}. Let $(X,\nu)$ be a
$(G,\mu)$-stationary space. For shortness of notation, let $h_\mu$ and
$h_\theta$ denote $h_\mu(X,\nu)$ and $h_\theta(X,\nu)$, respectively.
Recall the definitions of the maps $\rho$ and $\varphi$ given by
Eqs.~\ref{eq 3} and~\ref{eq 5}:
\begin{align*}
  \rho(g,x)=-\log\dfrac{dg^{-1}\nu}{d\nu}(x),\quad\quad\quad
  \varphi(g)=\int_X\rho(g,x)d\nu(x).
\end{align*}

\begin{lemma}\label{lemma nearly harmonic}
For any $g\in G$,
\begin{equation}\label{eq twisted function}
  \sum_{g_1\in G}\mu(g_1)\cdot\varphi(gg_1)  =\varphi(g)+h_\mu\,.
\end{equation}
\end{lemma}

Hence $\varphi$ is, in a sense, a ``nearly harmonic'' function, in
which the mean value property is corrected by a factor of $h_\mu$.

\begin{proof}
For a fixed $g\in G$, integrate the cocycle relation
\begin{equation*}
  \rho(gg_1,x)=\rho(g,g_1x)+\rho(g_1,x)
\end{equation*}
with respect to $g_1$ and $x$ to get
\begin{align*}
  \sum_{g_1\in G}\mu(g_1)\cdot\varphi(gg_1)=\sum_{g_1\in G}\mu(g_1)\cdot\int_X\rho(g,g_1x)d\nu(x)+h_\mu\,.
\end{align*}
Applying a change of coordinates and using the $\mu$-stationarity
of $\nu$, the sum in the right hand side of the above equality is
\begin{align*}
\sum_{g_1\in G}\mu(g_1)\cdot\int_X\rho(g,x)dg_1\nu(x)&=
\int_X\rho(g,x)d\left(\sum_{g_1\in G}\mu(g_1)\cdot g_1\nu\right)(x)\\
&=\int_X\rho(g,x)d\nu(x)\\
&=\varphi(g),
\end{align*}
and the claim is proved.
\end{proof}

Proceeding as in the proof of Lemma \ref{lemma gamma boundary}, one may define
the martingale $M_1,M_2,\ldots$ by
\begin{align*}
M_n=\varphi(Z_n)-n\cdot h_\mu.
\end{align*}
The expectation of each $M_n$ is equal to $\E{M_0}=\varphi(e)=0$.
Note that, when this martingale or its increments are uniformly bounded,
Theorem~\ref{thm abramov} is a consequence of the optional stopping theorem:
\begin{align*}
0=\E{M_0}=\E{M_{\htime}}=\E{\varphi(Z_\htime)-\htime\cdot h_\mu}=h_\theta-\E{\htime}\cdot h_\mu.
\end{align*}

However, this martingale is not bounded, and in general neither are
its increments. For example, when $\mu$ has full support and
$h_\mu\not=0$, then $|M_n-M_{n+1}|$ can be arbitrarily large.  Hence
the optional stopping theorem cannot be used to prove Theorem \ref{thm
  abramov}.  Instead, we replicate below a proof of the optional
stopping theorem, keeping account of an error term $R_n$. To show that
$R_n$ vanishes, we prove in Appendix~\ref{appendix probabilistic
  lemma} a lemma for the Markov chain $\{Y_n\}_{n=1}^\infty$ defined
in Section~\ref{section hitting time=index}.

Let $t_n=\P{\htime=n}$ be the probability that the return time to
$\Gamma$ is equal to $n$. Start by writing Eq.~\ref{eq twisted
  function} at $g=e$. Since $\varphi(e)=0$, this becomes
\begin{align*}
  h_\mu&=\sum_{g_1\in G} \mu(g_1) \cdot \varphi(g_1).
\end{align*}
Separating the sum into values of $g_1$ that are in $\Gamma$ and
values that are not, we get
\begin{align*}
  h_\mu &=\sum_{g_1\in \Gamma} \mu(g_1) \cdot \varphi(g_1) + \sum_{g_1
    \not
    \in \Gamma} \mu(g_1) \cdot \varphi(g_1) \\
  &=\theta^{(1)}(\varphi)+\sum_{g_1\not\in\Gamma} \mu(g_1) \cdot
  \varphi(g_1) .
\end{align*}
To each term $\varphi(g_1)$ with $g_1\not\in\Gamma$ apply Eq.~\ref{eq
  twisted function} again to obtain
\begin{eqnarray*}
  h_\mu&=&\theta^{(1)}(\varphi)+
  \sum_{g_1\not\in\Gamma} \mu(g_1)\left(\sum_{g_2\in G} \mu(g_2) \cdot \varphi(g_1g_2) -h_\mu\right) \\
  &=&\theta^{(1)}(\varphi)+
  \sum_{g_1\not\in\Gamma\atop{g_2\in G}}\mu(g_1) \cdot \mu(g_2) \cdot \varphi(g_1g_2) -h_\mu\cdot\sum_{g_1\not\in\Gamma}\mu(g_1)\\
  &=&\theta^{(1)}(\varphi)+
  \sum_{g_1\not\in\Gamma\atop{g_2\in G}} \mu(g_1)\cdot \mu(g_2) \cdot \varphi(g_1g_2)  -h_\mu\cdot(1-t_1)\\
  &=&\theta^{(1)}(\varphi)+\theta^{(2)}(\varphi)+
  \sum_{g_1\not\in\Gamma\atop{g_1g_2\not\in\Gamma}} \mu(g_1)\cdot
  \mu(g_2) \cdot \varphi(g_1g_2)
  -h_\mu\cdot(1-t_1).
\end{eqnarray*}
Then, using the fact that $\sum_{k\ge 1}t_k=1$, we arrive at
\begin{align*}
  h_\mu\cdot\left(\sum_{k\ge 1}t_k+\sum_{k\ge 2}t_k\right)=
  \theta^{(1)}(\varphi)+\theta^{(2)}(\varphi)+\sum_{g_1\not\in\Gamma\atop{g_1g_2\not\in\Gamma}}\mu(g_1)
  \cdot \mu(g_2) \cdot \varphi(g_1g_2)
\end{align*}
 Recalling Definition~\ref{def
  avoiding} of the avoiding sets $A_n$, we can repeat the above
argument to conclude that
\begin{equation}\label{eq 7}
h_\mu\cdot\sum_{j=1}^n\sum_{k\ge j}t_k=\sum_{k=1}^{n}\theta^{(k)}(\varphi)+
\sum_{(g_1,\ldots,g_n)\in A_n}\mu(g_1)\cdots\mu(g_n) \cdot \varphi(g_1\cdots g_n).
\end{equation}
Observe that
\begin{align*}
  \lim_{n\rightarrow\infty}\sum_{j=1}^n\sum_{k\ge
    j}t_k=\lim_{n\rightarrow\infty}\sum_{k=1}^nk \cdot t_k =\E{\htime},
\end{align*}
which equals $[G:\Gamma]$, by Theorem~\ref{thm kac}. Observe also that
\begin{align*}
  \lim_{n\rightarrow\infty} \sum_{k=1}^n\theta^{(k)}(\varphi) =
  \theta(\varphi) = h_\theta.
\end{align*}
Hence, if we take the limit of Eq.~\ref{eq 7} as $n$ goes to infinity,
we get
\begin{align*}
  h_\mu \cdot [G:\Gamma] = h_\theta+\lim_{n \rightarrow
    \infty}\sum_{(g_1,\ldots,g_n)\in A_n}\mu(g_1)\cdots\mu(g_n) \cdot
  \varphi(g_1\cdots g_n).
\end{align*}
The theorem will be proved if we show that the error term
\begin{align*}
  R_n &=\sum_{(g_1,\ldots,g_n)\in A_n} \mu(g_1)\cdots\mu(g_n) \cdot
  \varphi(g_1\cdots g_n)\\
  &= \E{\varphi(Z_n) \cdot \ind{\htime > n}}
\end{align*}
converges to zero as $n$ goes to infinity. To this purpose, we first
bound $\varphi(g_1\cdots g_n)$.

\begin{lemma}\label{lemma bounding phi}
  For any $g_1,\ldots,g_n\in G$,
  \begin{equation}\label{eq 8}
    \varphi(g_1\cdots g_n)\le-\sum_{k=1}^n\log\mu(g_k).
  \end{equation}
\end{lemma}

\begin{proof}
  We first show that for any $g\in G$, $x\in X$ and $n\ge 1$ it holds that
  \begin{equation}\label{eq 9}
    \rho(g,x)\le -\log \mu^n(g).
  \end{equation}
  This is stated in~\cite{kaimanovich1983random}.  Note that, since
  $\nu$ is $\mu$-stationary, it is also $\mu^{n}$-stationary, that is,
  $\sum_{g\in G}\mu^{n}(g)\cdot g\nu=\nu$.  Therefore, for any $x\in X$,
  \begin{align*}
    1= \dfrac{d\nu}{d\nu}(x) = \dfrac{d\left(\sum_{g\in
          G}\mu^{n}(g)\cdot g\nu\right)}{d\nu}(x) = \sum_{g\in
      G}\mu^{n}(g)\cdot\dfrac{dg\nu}{d\nu}(x)
  \end{align*}
  and since all the addends in the above sum are nonnegative, for any
  $g \in G$
  \begin{align*}
    \dfrac{dg^{-1}\nu}{d\nu}(x)\le\dfrac{1}{\mu^{n}(g^{-1})}\ \ \Longrightarrow\
    \ \rho(g,x)\ge -\log\dfrac{1}{\mu^n(g^{-1})}\, \cdot
  \end{align*}
  Now, by the cocycle property (Eq.~\ref{eq 4}) we have
  \begin{align*}
    \rho(g,x)=-\rho(g^{-1},gx)\le -\log\mu^n(g),
  \end{align*}
  thus establishing Eq.~\ref{eq 9}. Because this bound is independent of
  $x$, it implies
  \begin{align*}
  \varphi(g) = \int_X\rho(g,x)d\nu(x) \le -\log \mu^n(g).
  \end{align*}
  To conclude the proof, observe that
  \begin{align*}
    \mu^n(g_1\cdots g_n)\ge \mu(g_1)\cdots \mu(g_n)
  \end{align*}
  and so
  \begin{align*}
  \varphi(g_1\cdots g_n) \le -\log \mu^n(g_1\cdots g_n) \le
  -\sum_{k=1}^n\log \mu(g_k).
  \end{align*}
\end{proof}

Plugging this estimate in the error term $R_n$, we get
\begin{align*}
  R_n&=\E{\varphi(Z_n) \cdot \ind{\htime > n}}\\
  &\leq \E{-\left(\sum_{k=1}^n\log\mu(X_k)\right)\cdot\ind{\htime > n}}\\
  &= -\sum_{k=1}^n\E{\log\mu(X_k) \cdot \ind{\htime > n}}.
\end{align*}
Conditioning on $X_k$, we arrive at
\begin{align*}
  R_n &\leq -\sum_{k=1}^n\sum_{g \in G}\E{\log\mu(X_k) \cdot
    \ind{\htime > n}|X_k=g}\cdot\P{X_k=g}\\
  &= -\sum_{k=1}^n\sum_{g \in G}\mu(g)\log\mu(g)\cdot\P{\htime >
    n|X_k=g},
\end{align*}
since $\P{X_k=g} = \mu(g)$. Note that we are conditioning on the step
$X_k$ and not the position $Z_k$. By Lemma \ref{lemma probabilistic},
each of these conditional probabilities $\P{\htime > n|X_k=g}$ is
bounded by $e^{-Cn}$, for some $C>0$ independent of $n$ and $k$. Hence
\begin{align*}
  R_n\le\sum_{k=1}^n\sum_{g\in G}-\mu(g)\cdot\log\mu(g)\cdot
  e^{-Cn}=n\cdot e^{-Cn}\cdot H(\mu),
\end{align*}
which converges to zero as $n$ goes to infinity. This concludes the
proof of Theorem \ref{thm abramov}.

We finish this section by proving Corollary \ref{corollary}. Let
$(X,\nu)$ be the Furstenberg-Poisson boundary of $(G,\mu)$. Corollary
\ref{corollary poisson boundary} guarantees that $(X,\nu)$ is also the
Furstenberg-Poisson boundary of $(\Gamma,\theta)$ and thus, by Theorem
\ref{thm compare entropies},
\begin{align*}
  h_\mu(X,\nu)=h(G,\mu)\ \ \ \text{ and }\ \ \ h_\theta(X,\nu)=h(\Gamma,\theta).
\end{align*}
Corollary \ref{corollary} thus follows by plugging the above
equalities in Theorem \ref{thm abramov}.

\section{Acknowledgements}

The authors are thankful to The Weizmann Institute of Science for the
excellent atmosphere during the preparation of this manuscript and to Uri
Bader, Itai Benjamini, Hillel Furstenberg, Elchanan Mossel and Omri
Sarig for valuable comments and suggestions.

\appendix

\section{A lemma on the Markov chain
  $\{Y_n\}_{n=1}^\infty$}\label{appendix probabilistic lemma}

Let $\{M_n\}_{n=1}^\infty$ be an irreducible time-independent Markov
chain in the finite state space $S$. For each $x\in S$, let
\begin{align*}
  \htime_x=\min\{n\ge 1\,;\,M_n=x\}
\end{align*}
be the hitting time for $x$ and let $\Prob{x}{\,\cdot\,}$ denote the
probability in $\{M_n\}_{n=1}^\infty$ given that $M_0=x$. The result
below is standard in the theory of Markov chains (see, e.g., the
chapter on finite Markov chains in~\cite{levin2009markov}).

\begin{lemma}\label{lemma markov avoidance}
  Under the above condition, there exists $C>0$ such that
  \begin{align*}
    \Prob{x}{\htime_y>n}\le e^{-Cn}
  \end{align*}
  for all $x,y\in S$ and $n\ge 1$.
\end{lemma}

The lemma used to bound the error term $R_n$ in the proof of Theorem
\ref{thm abramov} can now be proved. We henceforth assume the
conditions and use the notation of Sections~\ref{section hitting
  time=index} and~\ref{section abramov formula}.

\begin{lemma}\label{lemma probabilistic}
  There exists $C>0$ such that
  \begin{align*}
    \P{\htime>n\,|\,X_k=g}\le e^{-Cn}
  \end{align*}
for all $n,k\ge 1$ and $g\in G$.
\end{lemma}

\begin{proof}
  Apply the previous lemma to the irreducible time-independent Markov
  chain $\{Y_n\}_{n=1}^\infty$ in the finite set $\Gamma\backslash G$
  as in Section \ref{section hitting time=index}, defined by the
  projection of the Markov chain $\{Z_n\}_{n=1}^\infty$ in $G$, to get
  $C>0$ such that
  \begin{align*}
    \Psub{g}{\htime > n} \leq e^{-2Cn}
  \end{align*}
  for $n \geq 1$ and $g \in G$, where $\mathbb{P}_g$ denotes the
  measure of $\mu$-random walks starting at $g$. We divide the proof into two cases.\\

  \noindent{\bf Case 1:} $k>n/2$. Because the event $\{\htime >n/2\}$
  is independent of $X_k$, we have
  \begin{align*}
    \P{\htime>n\,|\,X_k=g}
    &\le \P{\htime>n/2\,|\,X_k=g}\\
    &=\P{\htime>n/2}\\
    &\le e^{-Cn}.
  \end{align*}

  \noindent{\bf Case 2:} $k\le n/2$. By the law of conditional probabilities,
  $\P{\htime>n\,|\,X_k=g}$ is equal to
  \begin{align*}
  \sum_{\Gamma g_1\in\Gamma\backslash G}
  \P{\htime>n\,|\,Y_{n/2}=\Gamma g_1,X_k=g} \cdot
  \P{Y_{n/2}=\Gamma g_1\,|\,X_k=g}.
  \end{align*}
  We condition the first term on $\htime > n/2$, to get
  \begin{align*}
    \P{\htime>n\,|\,Y_{n/2}=\Gamma g_1,X_k=g} &\leq
    \P{\htime>n\,|\,\htime > n/2,Y_{n/2}=\Gamma g_1,X_k=g}\\
    &= \P{\htime>n\,|\,\htime > n/2,Y_{n/2}=\Gamma g_1}\\
    &= \Psub{g_1}{\htime > n/2},
  \end{align*}
  where the first equality follows from the Markov property. Hence
  \begin{align*}
    \P{\htime>n\,|\,X_k=g}& \le \sum_{\Gamma g_1\in\Gamma\backslash G}
    \Psub{g_1}{\htime>n/2}\cdot\P{Y_{n/2}=\Gamma g_1\,|\,X_k=g}\\
    &\le e^{-Cn}\sum_{\Gamma g_1\in\Gamma\backslash G}\P{Y_{n/2}=\Gamma g_1\,|\,X_k=g}\\
    &= e^{-Cn},
  \end{align*}
thus completing the proof of the lemma.
\end{proof}

\bibliography{abramov_formula}
\end{document}